\documentclass[reqno]{amsproc}
\usepackage[utf8]{inputenc}
\usepackage{amsfonts,amsmath,amssymb,amsxtra}
\usepackage{comment}
\usepackage{colonequals}
\usepackage[utf8]{inputenc}
\usepackage{hyperref}
\hypersetup{colorlinks=true,urlcolor=blue,citecolor=blue,linkcolor=blue}

\usepackage[all,cmtip]{xy}
\usepackage{float}

\usepackage[dvipsnames]{xcolor}
\usepackage[normalem]{ulem}
\usepackage{cleveref}

\usepackage{multirow,multicol}
\usepackage{diagbox}

\DeclareMathOperator{\Aut}{Aut}
\DeclareMathOperator{\Cls}{Cls}
\DeclareMathOperator{\Gen}{Gen}
\DeclareMathOperator{\GL}{GL}
\DeclareMathOperator{\GSp}{GSp}
\DeclareMathOperator{\GU}{GU}

\DeclareMathOperator{\diag}{diag}
\DeclareMathOperator{\isom}{isom}
\DeclareMathOperator{\M}{M}
\DeclareMathOperator{\new}{new}
\renewcommand{\O}{\text{O}}
\DeclareMathOperator{\impart}{Im}
\DeclareMathOperator{\old}{old}
\DeclareMathOperator{\rad}{rad}
\DeclareMathOperator{\SL}{SL}
\DeclareMathOperator{\SO}{SO}
\DeclareMathOperator{\Sp}{Sp}
\DeclareMathOperator{\std}{std}
\DeclareMathOperator{\Sym}{Sym}

\newcommand{\tT}{\intercal}
\numberwithin{equation}{section}

\newtheorem{thm}[equation]{Theorem}
\newtheorem{prop}[equation]{Proposition}
\newtheorem{conj}[equation]{Conjecture}

\theoremstyle{remark}

\newtheorem{rmk}[equation]{Remark}
\newtheorem{example}[equation]{Example}

\newcommand{\C}{\mathbb{C}}
\newcommand{\Q}{\mathbb{Q}}
\newcommand{\Z}{\mathbb{Z}}

\newcommand{\calH}{\mathcal{H}}

\newcommand{\frakp}{\mathfrak{p}}

\newcommand{\Lambdabar}{\overline{\Lambda}}

\newcommand{\defi}[1]{\textsf{#1}} 

\title[Paramodular forms database]{A database of paramodular forms from quinary orthogonal modular forms}

\author{Eran Assaf}
\address{Department of Mathematics, Dartmouth College, 6188 Kemeny Hall, Hanover, NH 03755, USA}
\email{assaferan@gmail.com}
\urladdr{\url{http://www.math.dartmouth.edu/~eassaf/}}

\author{Watson Ladd}
\address{Akamai Technologies}
\email{watsonbladd@gmail.com}

\author{Gustavo Rama}
\address{Facultad de Ingeniería, Universidad de la República, Montevideo, Uruguay}
\email{grama@fing.edu.uy}

\author{Gonzalo Tornar\'ia}
\address{Centro de Matemática, Universidad de la República, Montevideo, Uruguay}
\email{tornaria@cmat.edu.uy}

\author{John Voight}
\address{Department of Mathematics, Dartmouth College, 6188 Kemeny Hall, Hanover, NH 03755, USA}
\email{jvoight@gmail.com}
\urladdr{\url{http://www.math.dartmouth.edu/~jvoight/}}

\begin{document}

\begin{abstract}
We compute tables of paramodular forms of degree two and cohomological weight via a correspondence with orthogonal modular forms on quinary lattices.  
\end{abstract}

\thanks{Assaf and Voight were supported by a Simons Collaboration grant (550029, to Voight).  Rama and Tornar\'ia were partially supported by CSIC I+D 2020/651.}

\maketitle

\section{Introduction}

Number theorists have a longstanding tradition of making tables of modular forms
(for a brief history, see \cite[\S 2]{CMFs}), with myriad applications
to arithmetic and geometry.  Moving beyond $\GL_2$, there has been
substantial interest in similar catalogues of Siegel modular forms, as
computed first for $\Sp_4(\Z)$ by Kurokawa \cite{Kurokawa} and then
more systematically by Skoruppa \cite{Skoruppa}, Raum \cite{Raum}, and many
others.  Moving beyond trivial level, we find several interesting families of congruence subgroups of symplectic groups.  Among them, the paramodular
groups have recently received considerable interest, owing in part to
their agreeable theory of newforms \cite{gsp4new} as well as
applications in the Langlands program \cite{BrumerKramer,Gross}.
Direct computations of paramodular forms have focused on the more
troublesome case of (noncohomological) weight $2$
\cite{PoorYuen,BPY,PSY} (analogous to weight $1$ classical modular
forms), working with Fourier expansions using clever and sophisticated techniques. 

In this paper, we report on our computation of a moderately large database of
paramodular forms for $\GSp_4$ (i.e., degree $2$), but via a
complementary approach in weight $\geq 3$: we compute with algebraic
modular forms \cite{Gross:alg} on orthogonal groups of positive
definite quadratic forms in five variables.
Instead of working with Fourier expansions, we access only the
underlying systems of Hecke eigenvalues (enough for Galois
representations and $L$-function).
An explicit correspondence between orthogonal and paramodular forms
was first conjectured by Ibukiyama \cite{IK, Ibukiyama} and recently
proven by van Hoften \cite{vH}, R\"osner--Weissauer \cite{RW}, and
Dummigan--Pacetti--Rama--Tornar\'ia \cite{dprt:}, with a certain pesky asterisk (see \Cref{conj:dohweaklyeis}).  Our approach using quinary quadratic forms is analogous to the use of
ternary quadratic forms to compute classical modular forms introduced
by Birch \cite{ternary} (see also
\cite{tornariathesis,rama_msc,Hein:,htv}).
Our algorithms involve lattice methods, as described by
Greenberg--Voight \cite{GV} and further developed and investigated by
Hein \cite{Hein:}, Ladd \cite{Ladd}, and Rama \cite{rama_phd,rama_code}.

The tables computed here supersede data computed by Rama--Tornar\'ia \cite{RT} (squarefree level $N \leq 1000$): we compute in non-squarefull levels $N \leq 1000$, in higher weight, and with more Dirichlet coefficients (see \Cref{tab:my_label} and the surrounding details).  Our data is available online \cite{omfdata}, and it is being incorporated into the $L$-functions and Modular Forms Database (LMFDB) \cite{LMFDB}.

\subsection*{Acknowledgements}
The authors would like to thank Ariel Pacetti for helpful conversations and the anonymous referees for their feedback.  

\section{Background}


\subsection*{Siegel paramodular forms}
We begin with a brief review of some background and setting up notation. Let $g \in \Z_{\geq 1}$, let $V \colonequals \Q^{2g}$, and equip $V$ with the symplectic pairing
\begin{equation}
\begin{aligned}
\langle \,,\, \rangle \colon V \times V &\to \Q \\
\langle x, y\rangle &= x^{\tT} \begin{pmatrix} 0 & I \\ -I & 0 \end{pmatrix} y
\end{aligned}
\end{equation}
where $I$ is the $g \times g$-identity matrix.
Let $\GSp_{2g}(\Q)$ be the group of \defi{symplectic similitudes} of $V$, namely
\begin{equation}
\GSp_{2g}(\Q) \colonequals \{ (\alpha, \mu) \in \GL_{2g}(\Q) \times \Q^{\times} : \langle \alpha x,\alpha y \rangle = \mu \langle x,y\rangle \text{ for all $x,y \in V$}\}.
\end{equation}
Projection onto the second factor yields a character $\mu \colon \GSp_{2g}(\Q) \to \Q^{\times}$, whose kernel $\Sp_{2g}(\Q) \colonequals \ker \mu$ is the group of \defi{symplectic isometries} of $V$.  Let 
\[ \GSp_{2g}^+(\Q) \colonequals \{\alpha \in \GSp_{2g}(\Q) : \det(\alpha) > 0 \} < \GSp_{2g}(\Q). \]

The \defi{Siegel upper half-space} is
\begin{equation}
\calH_g \colonequals \{ z \in \M_g(\C) : \text{$z^\tT = z$ and $\impart(z) > 0$} \},
\end{equation}
the space of $g \times g$ complex symmetric matrices whose imaginary part is positive definite.
The group $\GSp_{2g}^+(\Q)$ acts on $\calH_g$ by
\[ \alpha z = \begin{pmatrix} a & b \\ c & d \end{pmatrix} z = (az+b)(cz+d)^{-1} \]
for $\alpha \in \GSp_{2g}^+(\Q)$ and $z \in \calH_g$.

Let $\rho \colon \GL_g(\C) \to \GL(W)$ be a representation of $\GL_g(\C)$ on a finite-dimensional $\C$-vector space $W$, and let $\Gamma \le \Sp_{2g}(\Q)$ be a discrete subgroup. The space $M_{\rho}(\Gamma$) of \defi{Siegel modular forms} of \defi{level} $\Gamma$ and \defi{weight} $\rho$ is the space of holomorphic functions $f \colon \calH_g \to W$ such that 
\begin{equation} \label{eqn:fgammaz}
f(\gamma z) = \rho(cz+d) f(z)  \qquad \text{ for all $\gamma = \begin{pmatrix} a & b\\ c & d \end{pmatrix} \in \Gamma$}
\end{equation}
together with the requirement that $f$ is holomorphic at the cusps of $\Gamma$ if $g= 1$.
When $g = 1$ and $\rho \colon \C^{\times} \to \C^{\times}$ is $x \mapsto x^k$, we write $M_k(\Gamma) = M_{\rho}(\Gamma)$; when $g = 2$ and $\rho = \rho_{k,j} \colonequals \det^k \otimes \Sym^j(\std)$, we write $M_{k,j}(\Gamma) = M_{\rho}(\Gamma)$. 

This paper is concerned with certain families of discrete subgroups for $g = 1$ and $g=2$.
Let $N$ be a positive integer. We write $\Gamma_0(N) \le \Sp_2(\Q) = \SL_2(\Q)$ $K(N) \le \Sp_4(\Q)$ for the following groups.
$$
\Gamma_0(N) \colonequals \begin{pmatrix} \Z & \Z \\ N \Z & \Z \end{pmatrix} \cap \SL_2(\Q), 
\quad 
K(N) \colonequals \begin{pmatrix}
 \Z & N\Z & \Z & \Z \\ 
 \Z & \Z & \Z &  N^{-1}\Z \\ 
\Z &  N\Z & \Z &  \Z \\
N \Z & N \Z & N \Z & \Z 
\end{pmatrix} \cap \Sp_4(\Q).
$$
The subgroup $K(N)$ is called the \defi{paramodular} subgroup of level $N$.

In particular, we will be interested in the space $M_{k,j}(K(N))$ of \defi{Siegel paramodular forms} of \defi{level} $N$ and \defi{weight} $(k,j) \in \Z_{\ge 0}^2$, and we write $M_k(N) = M_k(\Gamma_0(N))$ for the space of \defi{classical modular forms} of \defi{level} $N$ and \defi{weight} $k$. 

\subsection*{Hecke operators}
For $\alpha=\begin{pmatrix}
    a & b \\ c & d
\end{pmatrix} \in \GSp_4^{+}(\Q)$ (with $a,b,c,d \in \M_2(\Q)$), define a right action of $\GSp_4^+(\Q)$ on functions $f : \calH_2 \to \C$ 
\begin{equation}
(f \vert_{k,j} \alpha)(z) = \mu(\alpha)^{2k+j-3} \rho_{k,j}(cz + d)^{-1} f(\alpha z).
\end{equation}
In particular, the condition \eqref{eqn:fgammaz} for $\rho_{k,j}$ is equivalent to $(f \vert_{k,j} \alpha)(z)=f(z)$ for all $\alpha \in K(N)$ (since $\mu(\alpha)=1$).  

The \defi{Siegel operator} $\Phi \colon M_{k,j}(K(N)) \to M_{j+k}(N)$ is defined by
$$
\Phi(f)(z) = \lim_{t \to \infty} f \begin{pmatrix} z & 0 \\ 0 & it \end{pmatrix},
$$
and we say that $f \in M_{k,j}(K(N))$ is a \defi{cusp form} if $\Phi(f \vert_{k,j} \gamma) = 0$ for all $\gamma \in \Sp_4(\Q)$. We write $S_{k,j}(K(N))$ for the subspace of paramodular cusp forms.

For $p$ prime, let $\lambda_{p,1} \colonequals \diag(1,1,p,p) \in \M_4(\Z)$ and $\lambda_{p,2} \colonequals \diag(1,p,p^2,p) \in \M_4(\Z)$. For $i=1,2$, let $\{ \alpha_{p, i, r} \}_{r}$ be representatives for $K(N) \backslash K(N) \lambda_{p, i} K(N)$, so that
\begin{equation}
K(N) \lambda_{p,i} K(N) = \bigsqcup_{r} K(N) \alpha_{p,i,r};
\end{equation}
we define the \defi{Hecke operator} $T_{p,i} \colon M_{k,j}(K(N)) \to M_{k,j}(K(N)) $ by
$$
T_{p,i}(f) = \sum_{r} f \vert_{k,j} \alpha_{p,i,r}.
$$
The operators $T_{p,i}$ do not depend on the choice of representatives $\{\alpha_{p,i,r}\}$, and they restrict to operators on cusp forms $S_{k,j}(K(N))$ \cite[Proposition 5.2]{PSY18}, which we also denote $T_{p,i}$. 
The Hecke operators $T_{p,i}$ for $p \nmid N$ pairwise commute, and 
an \defi{eigenform} $f \in S_{k,j}(K(N))$ is a common eigenvector.

\subsection*{Paramodular oldforms and newforms}

The spaces $S_{k,j}(K(N))$ of paramodular cusp forms admit a theory of oldforms and newforms due to Roberts--Schmidt \cite{RS06}.  Let $p$ be a prime number.  Then there are three level-raising operators, defined as follows:
\begin{equation} \label{eqn:lvl1} 
\begin{aligned}
\theta_{N,p} \colon S_{k,j}(K(N)) &\to S_{k,j}(K(Np)) \\
\theta_{N,p}(f) &= f \vert_{k,j} \begin{pmatrix}
p &  &  & \\
 & p &  &  \\
 &  & 1 &  \\
 &  &  & 1 
\end{pmatrix}
\begin{pmatrix}
 1 & &  & \\
 &  & 1 &  \\
 &  -1 &  &  \\
 &  &  & 1 
\end{pmatrix} \\
&\qquad\qquad +
\sum_{c=0}^{p-1}  f \vert_{k,j} 
\begin{pmatrix}
p &  &  & \\
 & p &  &  \\
 &  & 1 &  \\
 &  &  & 1 
\end{pmatrix}
\begin{pmatrix}
 1&  &  & \\
 & 1 &  &  \\
  &  c & 1 &  \\
 &  &  & 1 
\end{pmatrix}
\end{aligned}
\end{equation}
and
\begin{equation} \label{eqn:lvl2}
\begin{aligned}
\theta_{N,p}' \colon S_{k,j}(K(N)) &\to S_{k,j}(K(Np)) \\
\theta_{N,p}'(f) &= f \vert_{k,j} 
\begin{pmatrix}
p &  &  & \\
 & 1 &  &  \\
 &  & 1 &  \\
 &  &  & 1/p 
\end{pmatrix} +
\sum_{c=0}^{p-1}  f \vert_{k,j} 
\begin{pmatrix}
 1&  &  & c/(pN) \\
 & 1 &  &   \\
  &  & 1 &  \\
 &  &  & 1 
\end{pmatrix}
\end{aligned}
\end{equation}
and finally
\begin{equation} \label{eqn:lvl3}
\begin{aligned}
\eta_{N,p} \colon S_{k,j}(K(N)) &\to S_{k,j}(K(Np^2)) \\
\eta_{N,p}(f) &= f \vert_{k,j} 
\begin{pmatrix}
 p &  &  & \\
 & 1 &  &  \\
  &  & 1 &  \\
 &  &  & p^{-1}
\end{pmatrix}
\end{aligned}
\end{equation}
These three operators $\theta_{N,p}$, $\theta_{N,p}'$, and $\eta_{N,p}$ commute with each other, with the level-raising operators for other primes, and with the operators $T_{\ell,i}$ for all $\ell \nmid Np$ \cite[\S 4]{RS06}.  

For $p \mid N$, the subspace of $p$-\defi{oldforms} in $S_{k,j}(K(N))$ is defined as the sum
\begin{equation}
\begin{aligned}
S_{k,j}^{\textup{$p$-old}}(K(N)) &\colonequals \theta_{N/p, p} S_{k,j}(K(N/p)) + \theta_{N/p, p}' S_{k,j}(K(N/p)) \\
&\qquad\qquad + \eta_{N/p^2, p} S_{k,j}(K(N/p^2))
\end{aligned}
\end{equation}
where the last term only occurs if $p^2 \mid N$; the space of \defi{oldforms} is then
\begin{equation}
S_{k,j}^{\textup{old}}(K(N)) \colonequals \sum_{p \mid N} S_{k,j}^{\textup{$p$-old}}(K(N)).
\end{equation}
The subspace of $p$-\defi{newforms} $S_{k,j}^{\textup{$p$-new}}(K(N)) \subseteq S_{k,j}(K(N))$ is the orthogonal complement of the subspace of $p$-oldforms
under the Petersson inner product, and the subspace of \defi{newforms} $S_{k,j}^\textup{new}(K(N))$
is the orthogonal complement of the subspace of oldforms.


\subsection*{Definite orthogonal modular forms}

As in the introduction, to compute paramodular forms in fact we compute certain orthogonal modular forms.  These spaces of orthogonal modular forms are again defined by a level $\Lambda$ and a weight $W$, as follows.

Let $V$ be a finite-dimensional $\Q$-vector space equipped with a positive definite quadratic form $Q \colon V \to \Q$ and $T \colon V \times V \to \Q$ the associated bilinear form. The \defi{orthogonal group} of $Q$ is the group of $F$-linear automorphisms of $V$ which preserve $Q$, namely 
$$
\O(Q) \colonequals \{ g \in \GL(V) : Q(gx) = Q(x) \textup{ for all $x \in V$}\};
$$
elements of $O(Q)$ are called \defi{isometries}.  Let
$$
\SO(Q) \colonequals \O(Q) \cap \SL(V) 
$$
(the subgroup of isometries that preserve an orientation). 

Let $\rho \colon \SO(Q) \to \GL(W)$ be a representation. We will sometimes refer to $\rho$ by the underlying space $W$. If $W$ is the irreducible representation of $\SO(Q)$ of highest weight $\lambda$, we write $W = W_\lambda$.

Let $\Lambda \subset V$ be a lattice, the $\Z$-span of a $\Q$-basis for $V$.  Suppose that $Q(\Lambda) \subseteq \Z$, in which case we say  $\Lambda$ is \defi{integral}.  For two lattices $\Lambda, \Lambda' \subseteq V$, if there is an isometry $g \in O(V)$ such that $\Lambda' = g (\Lambda)$, we say that $\Lambda$ and $\Lambda'$ are \defi{isometric}, and write $\Lambda \simeq \Lambda'$. 

Making analogous definitions, we have $\Lambda_p \colonequals \Lambda \otimes \Z_p$ a $\Z_p$-lattice in $V \otimes \Q_p$.  The \defi{genus} of $\Lambda$ is the set of lattices which are everywhere locally isometric to $\Lambda$, namely
\begin{equation}
\Gen(\Lambda) \colonequals \{ \Lambda' \subset V : \Lambda_p \simeq \Lambda'_p \text{ for all primes } p \}.
\end{equation}
There is a natural action of $\O(V)$ on $\Gen(\Lambda)$ and the \defi{class set} of $\Lambda$ is the set of global isometry classes
$$
\Cls \Lambda \colonequals \O(V) \backslash \!\Gen(\Lambda).
$$
By the geometry of numbers, the class set is finite; let $h = h(\Lambda) \colonequals \# \Cls \Lambda < \infty$. Let $\Lambda_1, \ldots, \Lambda_h$ be a set of representatives for $\Cls \Lambda$, and let 
\begin{equation}
\SO(\Lambda_i) \colonequals \{ g \in \SO(V) : g(\Lambda_i) = \Lambda_i \}
\end{equation}

The space of \defi{(special) orthogonal modular forms} of \defi{level} $\Lambda$ and \defi{weight} $W$ is the space 
\begin{equation}
M_W(\SO(\Lambda)) = \bigoplus_{i=1}^h W^{\SO(\Lambda_i)}.
\end{equation}
If $W = W_{a,b}$, we also write $M_{a,b}(\SO(\Lambda)) \colonequals M_{W}(\SO(\Lambda))$.
When $W$ is trivial, $(a,b) = (0,0)$, and $M_{W}(\SO(\Lambda)) \simeq \C^{\Cls \Lambda} \simeq \C^h$. 

For the correspondence with paramodular forms, we also need to define certain twists of these spaces by certain characters as follows. Let $p \mid N$ be prime, and let $R_p \colonequals \rad(\Lambda_p \otimes \Z / (2p)\Z)$ be the \defi{radical} of $\Lambdabar \colonequals \Lambda_p \otimes \Z /(2p)\Z$, where
\begin{equation}
\rad(\Lambdabar) \colonequals \{ x \in \Lambdabar : T(x,y) = 0 \text{ for all $y \in \Lambdabar$} \},
\end{equation}
(extending $T$ to $\Lambdabar$).
Since $g_p R_p = R_p$ for all $g_p \in \Aut(\Lambda_p)$, we may define a homomorphism $\vartheta_p \colon \Aut(\Lambda_p) \to \{ \pm 1 \}$ by $ \vartheta_p(g_p) \colonequals \det(g_p \vert_{R_p})$.
Given a unitary divisor $d \parallel N$ (so $d \mid N$ and $\gcd(d,N/d)=1$), we define the associated \defi{radical character} $\vartheta_d \colon \bigcap_{p \mid d} \Aut(\Lambda_p) \to \{ \pm 1 \} $ by $\vartheta_d = \prod_{p \mid d} \vartheta_p$ \cite[\S 7]{dprt:}.  Although $\vartheta_d$ is not defined on all of $\SO(V)$, as it is defined almost everywhere, using weak approximation we choose representatives $\Lambda_i$ of the class set so that it is defined on $\Aut(\Lambda_i)$ for all $i$.  

Finally, we write 
\begin{equation}
M_{a,b}(\SO(\Lambda), \vartheta_d) \colonequals M_{W_{a,b} \otimes \vartheta_d} (\SO(\Lambda)).
\end{equation}

\subsection*{Definite orthogonal methods for paramodular forms}

The relation between modular forms on $\SO(5)$
and automorphic forms on $\GSp(4)$ with trivial central character is predicted by Langlands functoriality. An
explicit correspondence was conjectured by Ibukiyama in
\cite{IK,Ibukiyama} involving two steps:
a correspondence between (para)modular
forms of $\GSp(4)$ and its compact twist $\GU(2,B)$,
where $B$ is a definite quaternion algebra;
and a correspondence between modular forms of $\GU(2,B)$ and those of
$\SO(Q)$ for a suitable chosen quinary quadratic form $Q$.
The first correspondence was proved by van Hoften \cite{vH} and
R\"osner--Weissauer \cite{RW}, and extended by
Dummigan--Pacetti--Rama--Tornar\'ia \cite{dprt:}, where the second
correspondence was also proved.

More precisely, using \cite[Theorem 9.9]{dprt:},
we can compute the space $S^{\text{new}}_{k,j}(K(N))$ of paramodular
newforms of level $N$ and weight $(k,j)$ under the following assumptions:
\begin{enumerate}
    \item There is a prime $p_0$ such that $p_0\parallel N$ (i.e., $p_0$ exactly divides $N$); and
    \item $k\geq 3$ and $j \in 2\Z_{\geq 0}$.
\end{enumerate}
(When $j$ is odd, the space is $\{0\}$.)  Suppose that these conditions hold.  Then \cite[Theorem 5.14]{dprt:} there exists a unique genus of  
positive definite integral quinary quadratic forms of (half-)discriminant $N$ with Eichler invariants $-1$ at $p_0$ and $+1$ at all other primes \cite[\S5]{dprt:}.
We choose one quadratic form $Q=Q_{N,p_0}$ in the genus, corresponding to a lattice $\Lambda$, noting that the corresponding space of
orthogonal modular forms is independent of this choice.

Following notation in \cite[\S 8]{dprt:} and for notational convenience, write $a \colonequals k+j-3$ and $b \colonequals k-3$, noting $a\equiv b\pmod2$.  Let $W_{a,b}$ be the corresponding weight representation of $\SO_5(\C)$; and for each $d\parallel N$, let $\vartheta_d$ be the associated radical character.
Let $M_{a,b}(\SO(\Lambda), \vartheta_d)$ for the space of orthogonal modular forms
for $Q_{N,p_0}$ with weight representation $W_{a,b}$ and character
$\vartheta_d$.  

Now let $S^{\text{$p_0$-new}}_{k,j}(K(N))_{\text{\bf(G)}} \subseteq S_{k,j}(K(N))$ be the space of $p_0$-new paramodular cusp forms of general type {\bf(G)}.  The condition \textbf{(G)} of general type concerns the type of the automorphic representation, following Schmidt \cite[\S 1.1]{Schmidt}: it excludes the forms of Saito--Kurokawa type \textbf{(P)} on both sides, the forms of Yoshida type \textbf{(Y)} appearing as orthogonal modular forms, and forms of Eisenstein type (those $f$ whose automorphic representation $\pi_f$ has local components $\pi_{f,p}$ isomorphic to a \emph{noncuspidal} irreducible constituent of a representation parabolically induced from a proper Levi subgroup).  See the section on detecting lifts below for further discussion.  Similarly, let $M_{a,b}(\SO(\Lambda), \vartheta_d)_{\textup{\bf(G)}} \subseteq M_{a,b}(\SO(\Lambda), \vartheta_d)$ be the subspace of forms of general type {\bf(G)}.

\begin{thm}[{\cite[Theorem 9.6]{dprt:}}] \label{thm:isomhecke}
There is an isomorphism of Hecke modules
\[    S^{\textup{$p_0$-new}}_{k,j}(K(N))_{\textup{\bf(G)}}
    \simeq
    \bigoplus_{d\parallel N}
    M_{a,b} (\SO(\Lambda), \vartheta_d)_{\textup{\bf(G)}}. \]
\end{thm}

In the isomorphism of \Cref{thm:isomhecke}, the paramodular forms corresponding to the summand $M_{a,b}(\SO(\Lambda),\vartheta_d)_{\text{\bf(G)}}$ have their Atkin-Lehner signs
$\varepsilon_p$ determined by $d$:
we have  $\varepsilon_{p_0}=-1$ if and only if $p_0\nmid d$, and
for primes $p\neq p_0$ we have $\varepsilon_p=-1$ if and
only if $p\mid d$ \cite[Theorem 10.1]{dprt:}.

In our analysis below, we can compute explicitly the forms of type \textbf{(P)} and 
\textbf{(Y)} using lifts.  To obtain only forms of type \textbf{(G)},
we need to identify those orthogonal modular forms of Eisenstein type,
which should be given as follows \cite[Proposition 9.8 and the comment
afterward]{dprt:}.

\begin{conj} \label{conj:dohweaklyeis}
The subspace of $M_{a,b} (\SO(\Lambda), \vartheta_d)$ spanned by forms of
Eisenstein type is the span of constant functions when $(a,b)=(0,0)$
(i.e., $(k,j)=(3,0)$) and is trivial otherwise.   
\end{conj}

\begin{prop}
\textup{\Cref{conj:dohweaklyeis}} holds when $N$ is squarefree.
\end{prop}

\begin{proof}
The same argument (comparing dimensions) in \cite[Proposition 9.8]{dprt:} applies.
\end{proof}

Thus, the correctness of our computations in nonsquarefree level are conditional on the truth of \Cref{conj:dohweaklyeis}: if this conjecture is false, it would mean that we have produced a counterfeit form which appears to be of type \textup{\bf(G)}.  We expect this never to happen, and we hope soon that this will be decisively resolved.


\section{Algorithmic comments} \label{sec:algcomments}


\subsection*{Algorithm overview}

Algorithms to compute with orthogonal modular forms using lattice methods were exhibited by Greenberg--Voight \cite{GV}; a recent over\-view with concrete examples is given by Assaf--Fretwell--Ingalls--Logan--Secord--Voight \cite[\S 3]{AFILSV}.  These algorithms take as input an integral, positive definite quadratic form on a lattice $\Lambda$ and compute the action of Hecke operators on spaces of functions on the class set of $\Lambda$, with values in a weight representation.  The Hecke operators are computed as $p$-neighbors, after Kneser, using an algorithm to test lattice isomorphism due to Plesken--Souvignier.   
For evaluating the Hecke operator $T_{p,1}$, the running time complexity is dominated by $ O(hp^3)$ isometry tests, where $h$ is the class number of the lattice; for $T_{p,2}$, it is dominated by $O(hp^4)$ isometry tests.

In this way, we can compute the Hecke operators for all good primes $p \nmid N$ and bad primes $p \parallel N$ \cite{RT}, giving the Euler factors at these primes.  
See the section below on Euler factors for bad primes $p^2\mid N$.

In order to compute all Dirichlet coefficients $a_n$ for $n \le M$ for the $L$-series attached to $f$, we need to compute the Hecke operators $T_{p,1}$ for $p \le M$ and the Hecke operators $T_{p,2}$ for $p \le \sqrt{M}$.  Hence the running time complexity for computing all of the first $M$ coefficients of the $L$-series is dominated by a total of $O\left(\displaystyle{\frac{hM^4}{\log M}}\right)$ isometry tests.

\subsection*{Implementation}

For reliability, we carried out and compared two separate implementations to compute the data, one in C and one in PARI/GP \cite{PARI/GP}. Eventually, these gave the same output.  The code can be found, respectively, at
\url{https://github.com/assaferan/omf5} and at \url{https://gitlab.fing.edu.uy/grama/quinary}. We have also used SageMath \cite{sage} for auxiliary operations on the data we produced.  

\begin{rmk}
In \cite{AFILSV}, the authors present a Magma \cite{Magma} implementation, to be found at \url{https://github.com/assaferan/ModFrmAlg}, which agrees with the other two.  It is more general (working in arbitrary dimension and over a totally real base field); our implementation is specialized, and more efficient for this case.
\end{rmk}

\subsection*{Algorithmic improvements}

In order to make the calculations of Hecke operators mentioned in the previous section more efficient, it is possible to take advantage of the action of the isometry group of the lattice. Indeed, two $p$-isotropic vectors in the same orbit of the isometry group $\Aut(\Lambda)$ will produce the same target lattice when applying the $p$-neighbor relation. It is also possible to save on isometry testing via taking the first few entries of the theta series for the lattice.  

To carry out this idea (which has been observed before), we present a few algorithmic improvements to further speed up the computation by a constant factor.

\begin{enumerate}
\item Instead of precomputing all the orbits of isotropic vectors under $\Aut(\Lambda)$, we order the isotropic vectors in a lexicographical order. Given a $\Z_p$-isotropic vector $v$, we compute its orbit under $\Aut(\Lambda)$, and proceed with the computation only when it is the minimal vector in its orbit. 

Using the automorphism group speeds up the running time by the average size of an orbit, hence by a factor of $\#\Aut(\Lambda)$. Note that even over all rank $5$ lattices $\Lambda$, the sizes $\#\Aut(\Lambda)$ remain bounded, hence this is only a constant factor improvement. Precomputation of the orbits \cite[\S 3]{AFILSV} also yields this improvement, but at the cost of using $O(p^3)$ memory for $T_{p,1}$ and $O(p^4)$ memory for $T_{p,2}$; here we reduce these memory costs without performing additional isometry tests or short vector computations.

\item We precompute the automorphism group of all lattices in the genus, and their conjugations into a single quadratic space, saving the cost of conjugation when computing the spinor norm. The factor that this improvements yields depends on the ratio $T_{\text{conj}} / T_{\theta}$ between the time spent on conjugation of matrices and the time spent on computing short vectors. In practice, in our C implementation, this improves the running time by an average factor of $1.3$. 


\item If we have a list of the lengths of short vectors of the lattice we need only check for isometry with those members of the genus with the same list. For a given genus we measure the cost $T_{\isom}$ of isometry testing and the cost $T_{\theta}(B)$ of computing the short vectors of length up to $B$, and we choose $B$ that optimizes the total cost $\alpha(B) T_{\isom} +  T_{\theta}(B)$, where $\alpha(B)$ is the average number of collisions, with the lattices in the genus averaged by the size of their automorphism groups.
This follows from the fact that the frequency of appearance of a lattice $\Lambda$ as a $p$-neighbor is inversely proportional to $\#\Aut(\Lambda)$, as proven by Chenevier \cite{Chenevier}. 
In practice, in our C implementation, this yields an average speed up factor of $6.84$. 
\end{enumerate}


\begin{example}
For the case $N = 61$ with trivial weight, in the C implementation, it takes $113$ s to compute all eigenvalues for $p<100$ without these improvements. With all of them put together it takes $10.28$ s, giving a speedup factor of $11$. 
\end{example}

\begin{example}
For the case $N = 61$, in the Pari/GP implementation, it takes $3881$ s to compute a row for the space of quinary modular forms for the representation  $(a,b)=(1,0)$ (corresponding to $(k,j)=(4,0)$) and the character $\theta_{61}$ without the isometry improvement.
With the isometry improvement it took $451$ s for a row that corresponds to a quadratic form with $96$ automorphisms. And $859$ s for a row that corresponds to a quadratic form with $12$ automorphisms. This gives us a speedup factor of $8.6$ and a factor of $4.5$ respectively.
\end{example}

\begin{example}
Again for the case $N = 61$, in the Pari/GP implementation, it takes $4440$ s to compute all a row for the space of quinary modular forms for weight $(a,b)=(1,1)$ and the character $\theta_{61}$, without the isometry improvement.
With the isometry improvement it took $779$ s for a row that corresponds to a quadratic form with $96$ automorphisms. And $1188$ s for a row that corresponds to a quadratic form with $12$ automorphisms. This gives us a factor of $5.7$ and a factor of $3.7$, respectively.
\end{example}

\subsection*{Detecting lifts}

Among the orthogonal forms we compute
we have forms that arise
from lifts (endoscopy),
just like Eisenstein series in the classical case.
Since those
lifts are classified and can be computed in other ways, we focus on
what remains on computing newforms of type \textbf{(G)}.

To discard forms of type \textbf{(P)} corresponding to Saito--Kurokawa lifts,
a single good Euler factor is enough: a form of type \textbf{(P)}
will have a Satake parameter $p^{1/2}$
that cannot otherwise appear by the Ramanujan
conjecture for non-CAP forms \cite{Weissauer-Ramanujan}.
The multiplicity for forms of type \textbf{(P)}
is given by \cite[Theorem 5.5.9]{gsp4new}; we do not need this as
the stated criterion will filter out old and new Saito--Kurokawa lifts
alike.


To discard forms of type \textbf{(Y)} corresponding to Yoshida lifts, it suffices to find
that a single good Euler factor is irreducible.
If all the computed degree 4 Euler factors are product of two degree 2
factors, we look in tables to conjecturally identify the form
as a Yoshida lift, as we know exactly which Yoshida
lifts should appear \cite[Proposition 9.1]{dprt:}; the radical
character is determined by the Atkin-Lehner signs which are given by
\cite[Theorem 10.1 (3)]{dprt:}.
The multiplicity of each Yoshida lift is given by
\cite[Theorem 7.5.6]{gsp4new} with Atkin-Lehner signs as in
\cite[Table (5.49)]{gsp4new}; those are local results so we have to
take the product over primes with Eichler invariant +1 (for primes with
Eichler invariant $-1$---i.e., $p_0$---Yoshida lifts are always new and
have local multiplicity one).
To identify
the lifts, a simple inspection and comparison of the traces often
suffices, otherwise we use the complete Euler factors.


\subsection*{Newforms and oldforms}  

In all cases, we have a good guess as to what is new and old;
computing inductively we can verify that a form is not an oldform.
However, oldforms may appear in the space of orthogonal modular forms
with multiplicity.

Using \cite[Theorem 7.5.6]{gsp4new} we can compute the
multiplicity for each oldform with Atkin-Lehner signs given by
\cite[Table (5.49)]{gsp4new}. Again we must take the product over
primes with Eichler invariant $+1$; for primes with Eichler invariant
$-1$ the forms are always new with multiplicity one and the Atkin-Lehner
sign is changed (see \cite[Theorem 10.1 (1)]{dprt:}).

With these multiplicities, we can certify the oldform and newform
decomposition by computing more Hecke eigenvalues until the dimension
of a candidate oldspace with given eigenvalues matches the multiplicity.

\section{Running the calculation}


\subsection*{Example}

We begin with an example to illustrate the calculations we performed.





Consider the space $S_{3,0}(K(312))$. 
Since $312 = 2^3 \cdot 3 \cdot 13$, we choose $p_0 = 13$, and produce a lattice $\Lambda$ with (half-)discriminant $312$ and Eichler invariant $-1$ at $13$ and $+1$ at $2$ and $3$. Explicitly we take $\Lambda=\Z^5$ equipped with the quadratic form having the following Gram matrix:
$$
\begin{pmatrix}
   2  & 0 & 1 & 0 & 1 \\
   0  & 2 & -1 & -1 & 1 \\
   1 & -1 & 4 & -1 & 0 \\
   0 & -1 & -1 & 6 & 0 \\
   1 & 1 & 0 & 0 & 12
\end{pmatrix}
$$
A quick computation shows that the class set of $\Lambda$ has cardinality $15$.  Computing the spaces $M_{0,0} (\SO(\Lambda), \theta_d)$ for squarefree $d \mid 312$ we find that their dimensions are given as in Table~\ref{tab:example_subspace_dims}.

\begin{table}[ht]
\hfuzz=2pt\relax
\centering
\begin{tabular}{l||c|c|c|c|c|c|c|c||c}
\hfill$d$ & 1 & 2 & 3 & 6 & 13 & 26 & 39 & 78 & Total \\
\hline\hline
{\bf(G)}-new 
& 1 & 0 & 0 & 5 = 3 + 1 + 1 & 0 & 1 & 4 = 3 + 1 & 0 & 11 \\
{\bf(G)}-old 
& 1 & 1 & 1 & 1 & 1 & 2 & 0 & 1 & 8 \\
{\bf(P)}-new 
& 2 & 0 & 0 & 3 & 0 & 2 & 2 & 0 & 9 \\
{\bf(P)}-old 
& 8 & 0 & 0 & 2 & 0 & 7 & 2 & 0 & 19 \\
{\bf(Y)}-new 
& 1 & 0 & 0 & 1 & 0 & 1 & 3 & 0 & 6\\
{\bf(Y)}-old
& 1 & 1 & 0 & 0 & 1 & 1 & 0 & 0 & 4\\
\hline \hline
Total
& 14 & 2 & 1 & 12 & 2 & 14 & 11 & 1 & 57\\
\end{tabular}
\caption{Dimensions of subspaces of $M_{0,0}(\SO(\Lambda), \theta_d)_{(\text{A})}$ for $d \mid 312$}
\label{tab:example_subspace_dims}
\end{table}

As explained in the previous section, we calculate the space of Yoshida lifts:
\begin{equation}
\begin{aligned}
M_{0,0}^{\new}(\SO(\Lambda))_{\text{\bf(Y)}} &\simeq 
(S_2^{\new}(\Gamma_0(26)) \otimes S_4^{\new}(\Gamma_0(12))) \\
& \qquad \qquad \oplus (S_2^{\new}(\Gamma_0(39)) \otimes S_4^{\new}(\Gamma_0(8))) \\
& \qquad \qquad \oplus (S_2^{\new}(\Gamma_0(52)) \otimes S_4^{\new}(\Gamma_0(6))),
\end{aligned}
\end{equation}
leading to the corresponding dimension counts in Table~\ref{tab:example_subspace_dims}. Furthermore, the only Yoshida lifts that occur as oldforms are the images of the forms in 
$$
M_{0,0}^{\new}(\SO(\Lambda_{156}))_{\text{\bf(Y)}} \simeq S_2^{\new}(\Gamma_0(26)) \otimes S_4^{\new}(\Gamma_0(6))
$$
under the level-raising operators $\theta_2, \theta_2' \colon S_{3,0}(K(156)) \to S_{3,0}(K(312))$.

Similarly, we find that
$$
M_{0,0}^{\new}(\SO(\Lambda))_{\text{\bf(P)}} \simeq 
S_4^{\new,-}(\Gamma_0(312)) \oplus S_4^{\new,+}(\Gamma_0(24)),
$$
(plus and minus signs for the Atkin--Lehner involution), and
\begin{equation}
\begin{aligned}
M_{0,0}^{\old}(\SO(\Lambda))_{\text{\bf(P)}} &\simeq 
\bigoplus_{d \mid 24, d \ne 24} S_4^{\new,-}(\Gamma_0(13d)) \oplus S_4^{\new,+}(\Gamma_0(d)) \\
&\qquad\qquad \oplus 
\bigoplus_{d \mid 6} S_4^{\new,-}(\Gamma_0(13d)) \oplus S_4^{\new,+}(\Gamma_0(d));
\end{aligned}
\end{equation}
this yields further dimension counts in Table~\ref{tab:example_subspace_dims}.

Finally, using our precomputed data from lower levels, we find that 
\[ \dim S_{3,0}(K(156))_{\text{\bf(G)}} = 3, \quad \dim S_{3,0}(K(104))_{\text{\bf(G)}} = 1. \] 
Using the more precise data of their Atkin-Lehner eigenvalues and applying the corresponding level-raising operators, we find that they contribute to an $8$-dimensional subspace, distributed between the different subspaces as described in the second row of Table~\ref{tab:example_subspace_dims}.

The subspace orthogonal to all of the above is the space of Siegel paramodular newforms of type $(\textbf{G})$ \cite[Theorem 9.9]{dprt:}, given in the top row.  For these, we have also included in the table the decomposition of the subspaces into Hecke irreducible submodules, i.e., Galois orbits of paramodular newforms. 

It then follows that $\dim S_{3,0}^{\new} (K(312))_{\text{\bf{(G)}}} = 11$, and the space is composed of $7$ Galois orbits of newforms, out of which $5$ newforms have rational eigenvalues and the $2$ remaining orbits have dimension $3$. 
Table~\ref{tab:example_subspace_dims_paramodular} gives the dimensions of various subspaces of $S_{3,0}(K(312))$.

\begin{table}[ht]
\hfuzz=2pt\relax
\centering
\begin{tabular}{c||c|c||c}
type & new & old & total \\
\hline
\bf{(G)} & 11 & 8 & 19 \\
\hline
\bf{(P)} & 9 & 19 & 28 \\
\hline \hline
Total & 20 & 27 & 47 \\
\end{tabular}
\caption{Dimensions of subspaces of $ S_{3,0}(K(312))$ }
\label{tab:example_subspace_dims_paramodular}
\end{table}

Table~\ref{tab:example_al_dimensions} gives the dimensions of the cuspidal new subspaces of automorphic type {\bf{(G)}} with specified eigenvalues for the Atkin-Lehner operators and the Fricke involution.

\begin{table}[ht]
\hfuzz=2pt\relax
\centering
\begin{tabular}{c|c|c||c|c}
2 & 3 & 13 & Fricke & dimension \\
\hline
$+$ & $+$ & $-$ & $-$ & $1$ \\
\hline
$+$ & $-$ & $+$ & $-$ & $4$ \\
\hline
$-$ & $+$ & $+$ & $-$ & $1$ \\
\hline
$-$ & $-$ & $-$ & $-$ & $5$ \\
\hline \hline
\multicolumn{3}{c||}{Plus space} & $+$ & $0$ \\
\multicolumn{3}{c||}{Minus space} & $-$ & $11$ \\
\end{tabular}
\caption{Dimensions of Atkin--Lehner subspaces of $ S_{3,0}^{\new}(K(312))_{\text{\bf{(G)}}} $ }
\label{tab:example_al_dimensions}
\end{table}

Finally, since we are able to compute the Hecke eigenvalues at good primes for each of the newforms, we can decompose $ S_{3,0}^{\new}(K(312))_{\text{\bf{(G)}}} $ to newform subspaces. Table~\ref{tab:example_eigenvalues} lists the Galois orbits of the Hecke eigenforms in this space, giving rise to such a decomposition.

\begin{table}[ht]
\hfuzz=2pt\relax
\centering
\begin{tabular}{c|c||r|r|r|r||c|c|c}
 \multirow{2}{*}{dimension} & \multirow{2}{*}{field} & \multicolumn{4}{c||}{Traces} & 
 \multicolumn{3}{c}{A-L signs} \\
 \cline{3-9}
& &  $a_5$ & $a_7$ & $a_{11}$ & $a_{17}$ & $2$ & $3$ & $13$ \\
\hline \hline
$1$ & $\Q$ & $-1$ & $-13$ & $-6$ & $63$ & $+$ & $+$ & $-$ \\
$1$ & $\Q$ & $-11$ & $3$ & $-16$ & $3$ & $+$ & $-$ & $+$ \\
$1$ & $\Q$ & $1$ & $-15$ & $14$ & $135$ & $-$ & $+$ & $+$ \\
$1$ & $\Q$ & $2$ & $-6$ & $-52$ & $44$ & $-$ & $-$ & $-$ \\
$1$ & $\Q$ & $-13$ & $-3$ & $-4$ & $-37$ & $-$ & $-$ & $-$ \\
$3$ & \href{http://www.lmfdb.org/NumberField/3.3.961.1}{\textsf{3.3.961.1}} & $-1$ & $-25$ & $-24$ & $-71$ & $+$ & $-$ & $+$ \\
$3$ & \href{http://www.lmfdb.org/NumberField/3.3.961.1}{\textsf{3.3.961.1}} & $-12$ & $28$ & $2$ & $-90$ & $-$ & $-$ & $-$ \\
\end{tabular}
\caption{Decomposition of $ S_{3,0}^{\new}(K(312))_{(\text{\bf{(G)}})} $ into newform subspaces}
\label{tab:example_eigenvalues}
\end{table}
Here, the field \textsf{3.3.961.1} is given by its LMFDB label: it is defined by a root of the polynomial $x^3-x^2-10x+8$.

\subsection*{Data and running time}

The data is available online \cite{omfdata}.  We computed the spaces of paramodular forms of level $N$ and weight $(k,j)$, the Hecke eigenforms and the eigenvalues of the Hecke operators in the following ranges:
\begin{itemize}
\item $(k,j)=(3,0)$, $D=N \leq 1000$, good $T_{p,i}$ with $p^i<200$
\item $(k,j)=(4,0)$, $D=N \leq 1000$, good $T_{p,1}$ with $p<100$, good $T_{p,2}$ with $p<30$
\item $(k,j)=(3,2)$, $D=N \leq 500$, good $T_{p,1}$ with $p<100$, good $T_{p,2}$ with $p<30$
\end{itemize}



The total counts are summarized in Table \ref{tab:my_label}.

\begin{table}[ht]
\hfuzz=2pt\relax
\centering
\begin{tabular}{c||c|c|c||c|c|c}
\multirow{2}{*}{$(k,j)$} & \multicolumn{3}{c|}{Newspaces} & \multicolumn{3}{c}{Newforms}\\\cline{2-7}
& sqfree $N$ & nonsqfree $N$ & total & sqfree $N$ & nonsqfree $N$ & total\\\hline\hline
$(3, 0)$ & 2\,764 & 4\,820 & 7\,584 & 52\,181 & 23\,853 & 76\,034\\
$(3, 2)$ & 1\,363 & 3\,072 & 4\,435 & 72\,551 & 29\,226 & 101\,777\\
$(4, 0)$ & 2\,856 & 7\,783 & 10\,639 & 287\,974 & 132\,380 & 420\,354\\
\end{tabular}
\caption{Newspace and newform data computed (abbreviating squarefree to ``sqfree'')}
\label{tab:my_label}
\end{table}

For squarefree levels, we use formulas from Ibukiyama--Kitayama \cite{IK} and Ibukiyama \cite{Ibukiyama} to double-check our results.


We compare the performance of the two implementations (Pari and C) by comparing the computation time of a single row of the matrix representing the Hecke operator $T_{p,1}$ on the spaces $S_{3,0}(K(N))$ for various values of $N$ and $p$. We note that this step is the bottleneck of the computation, hence this comparison provides ample evidence for the differences in performance as the parameters scale.

\begin{table}[ht]
\centering
\begin{tabular}{l||c|c||c|c||c|c||c|c}
\backslashbox{$p$}{$N$} & \multicolumn{2}{c|}{61} & \multicolumn{2}{c|}{167} & \multicolumn{2}{c|}{262} & \multicolumn{2}{c}{334} \\
\hline
 & C & Pari & C & Pari & C & Pari & C & Pari \\
\hline \hline
$31$  & 0.05 s & 1.27 s & 0.45 s  & 1.31 s & 0.37 s  & 1.58 s & 0.49 s  & 2.46 s \\
$41$  & 0.10 s & 2.86 s & 0.80 s  & 2.84 s & 0.94 s  & 3.56 s & 1.02 s  & 5.53 s\\
$101$ & 1.09 s & 41.3 s & 6.71 s  & 41.1 s & 7.86 s  & 50.8 s & 12.6 s & 79.8 s\\
$127$ & 2.00 s & 81.8 s & 10.1 s & 81.4 s & 14.0 s & 100 s & 20.1 s & 158 s\\
$199$ & 6.49 s & 312 s & 37.16 s & 312 s & 42.16 s & 384 s & 75.06 s & 607 s\\
\end{tabular}
\caption{Timings for computing $T_{p,1}$}
\label{tab:timings_table_1}
\end{table}


For Hecke irreducible spaces of dimension $\geq 20$, we only store the traces; to avoid painful calculations in an extension field, we compute modulo a prime $\frakp$ of degree $1$ which is large enough, and reconstruct.

The total data takes approximately 200 MB of disk space and took a total of $4575$ hours of CPU time on a standard processor (Intel Xeon Gold 6240 CPU, 2.60GHz). Computing the eigenvalues for $p < 200$ took a total of $13869$ hours of CPU time on a standard processor (AMD Ryzen Threadripper 2970WX CPU, 2.20 GHz) equipped with Ubuntu 22.04.2 operating system, achieving maximal memory usage of 303276 KB. 

\subsection*{Bad Euler factors} 

When $p \parallel N$, the local Euler factor can be computed via neighbors (see  \cref{sec:algcomments}).  When $p^2\mid N$, the local Euler factor has degree at most 2 and is given explicitly by Roberts--Schmidt \cite[Theorem 7.5.3]{gsp4new} in terms of eigenvalues for a pair of Hecke operators which correspond to $p$ and $p^2$-neighbors \cite[Proposition 8.5]{dprt:} (extended in the same way for $p\neq p_0$).

The first such bad factor occurs at level $76=2^2 \cdot 19$: the bad Euler factors are
\begin{align*}
L_2(f_{76}, X) &= 1 + 5X + 2^3X^2 \\
L_{19}(f_{76},X) &= (1 + 19X)(1-50X+19^3X^2).
\end{align*}
(Since $19 \parallel 76$, we compute the latter without any guesswork.)
The next one occurs at level $96 = 2^5 \cdot 3$, with 
\begin{align*}
L_2(f_{96}, X) &= 1 + 4X + 2^3X^2 \\
L_3(f_{96}, X) &= (1-3X)(1+8X+3^3X^2).
\end{align*}
We used our truncated Dirichlet series to check that the functional equation is satisfied up to 22 decimal digits of precision in each case, using both Magma and Pari/GP.  
(In practice, for $p^2\mid N$ we found easier to guess the appropriate Hecke operator via reconstruction by checking that the functional equation for the $L$-function is satisfied---ruling out all possibilities but one.)

In order to rigorously verify that the functional equation is satisfied to a given precision, one would need the first $C \sqrt{N}$ embedded Dirichlet coefficients (with sufficient precision) for a constant $C$ (see below), carrying out all floating-point calculations using rigorous error bounds and interval arithmetic. A generic library to carry out such computations, due originally to David Platt \cite{lfunc}, has been developed; it includes functionality to estimate the constant $C$ as a function of the motivic data.  However, for $L$-functions of Siegel paramodular forms, it yields an estimated constant of $C \approx 138.84$, meaning one needs the Dirichlet coefficients for $p^n < 1211$ for $f_{76}$ and for $p^n < 1361$ for $f_{96}$. This is in stark contrast to the small constant $C = 5$ which suffices for genus $2$ curves \cite{genus2} or $C = 0.08k \log k + 24$ sufficient for classical modular forms of weight $k$ \cite{CMFs}.  Therefore, at the moment we satisfy ourselves with the heuristic of checking the functional equation for the truncated Dirichlet series that we have computed.

\subsection*{Statistics}

Having compiled such a database with substantial amount of information, using the LMFDB \cite{LMFDB}, we are able to provide some arithmetic statistics of interest, examples of which we present in Table~\ref{tab:dimension_by_level} and Table~\ref{tab:aut_type_by_level}, illustrating the distribution of the dimension of the Galois orbit of a newform of type $\bf{(G)}$, and the distribution of the automorphic type, respectively, as the level increases.

\begin{table}[ht]
\hfuzz=6pt\relax
\centering
\begin{tabular}{c||r|r|r|r|r|r|r||r}
\multirow{2}{*}{level} & \multicolumn{8}{c}{dimension}\\
\cline{2-9}
& 1 & 2 & 3 & 4 & 5-10 & 11-20 & 21-266 & Total \\
\hline
\hline
\multirow{2}{*}{1-100} & 14 & 1 &  & & & & & 15 \\
& 0.18\% & 0.01\% & & &  & &  & 0.20\% \\ 
\hline
\multirow{2}{*}{101-300} & 258 & 132 & 32 & 42 & 59 & 11 & 1 & 535 \\
& 3.40\% & 1.74\% & 0.42\% & 0.55\% & 0.78\% & 0.15\% & 0.01\% & 7.05\% \\ 
\hline
\multirow{2}{*}{301-600} & 864 & 384 & 183 & 156 & 358 & 164 & 155 & 2264 \\
& 11.40\% & 5.06\% & 2.41\% & 2.06\% & 4.72\% & 2.16\% & 2.04\% & 29.85\% \\ 
\hline
\multirow{2}{*}{601-999} & 1436 & 693 & 399 & 275 & 684 & 474 & 809  & 4770 \\
& 18.93\% & 9.14\% & 5.26\% & 3.63\% & 9.02\% & 6.25\% & 10.67\% & 62.90\% \\ 
\hline
\hline
\multirow{2}{*}{Total} & 2572 & 1210 & 614 & 473 & 1101 & 649 & 965 & 7584 \\
& 33.91\% & 15.95\% & 8.10\% & 6.24\% & 14.52\% & 8.56\% & 12.72\% &  \\ 
\end{tabular}
\caption{Distribution of levels and dimensions in weight $(3,0)$}
\label{tab:dimension_by_level}
\end{table}

\begin{table}[ht]
\hfuzz=2pt\relax
\centering
\begin{tabular}{c||r|r|r||r}
\multirow{2}{*}{level} & \multicolumn{4}{c}{automorphic type}\\
\cline{2-5}
& \bf{(G)} & \bf{(P)} & \bf{(Y)} & Total \\
\hline
\hline
\multirow{2}{*}{1-100} & 15 & 169 & 11 & 195 \\
& 0.10 \% & 1.15 \% & 0.07 \% & 1.29 \% \\ 
\hline
\multirow{2}{*}{101-300} & 535 & 787 & 166 & 1488 \\
& 3.53 \% & 5.19 \% & 1.09 \% &  9.81 \% \\ 
\hline
\multirow{2}{*}{301-600} & 2264 & 1613 & 709 & 4586 \\
& 14.93 \% & 10.64 \% & 4.68 \% & 30.25 \% \\ 
\hline
\multirow{2}{*}{601-999} & 4770 & 2476 & 1647 & 8893 \\
& 31.46 \% & 16.33 \% & 10.86 \% &  58.65 \%  \\ 
\hline
\hline
\multirow{2}{*}{Total} & 7584 & 5045 & 2533 & 15162 \\
& 50.02 \% & 33.27 \% & 16.71 \% &  \\ 
\end{tabular}
\caption{Distribution of automorphic type in weight $(3,0)$}
\label{tab:aut_type_by_level}
\end{table}

\subsection*{Interesting examples}
When compiling a large amount of modular forms in a database, it is easy to view them as a whole. However, each form has its own unique features and set of characteristics that make them different from any other form.
We seize the opportunity to point out several specific forms of interest, and refer to their rich history in the literature and in previous works. We focus on weight $(k,j) = (3,0)$ in this section.

\begin{itemize}
    \item The smallest level in which a paramodular form of weight $(3,0)$ and automorphic type $ \bf{(G)} $ occurs is $61$. This form has Hecke eigenvalues $a_2 = -7, a_3 = -3, a_5 = 3, a_7 = -9, \ldots $, and was first described in \cite{PoorYuen}.
    \item The smallest level in which there is a paramodular newform of weight $(3,0)$ with negative sign in the functional equation is $167$. This form has Hecke eigenvalues $a_2=-8, a_3=-10, a_5=-4, a_7=-14, \ldots$. It was found in \cite{RT}, paving the way for the correspondence presented in \cite{dprt:}.
    \item The first level to exhibit a non-rational paramodular newform of type $ \bf{(G)} $ is $97$. The Hecke field is $\Q(\sqrt{5})$, and the first Hecke eigenvalues are $a_2 = -5+\varphi , a_3 = -4 \varphi, a_5 = -8 + 12 \varphi, a_7 = 9 - 2 \varphi, \ldots$, where $\varphi = \frac{1 + \sqrt{5}}{2}$.
    \item The paramodular newform of type $ \bf{(G)} $ with the largest dimension of its Galois orbit in our database occurs in level $997$, and its dimension is $266$. 
    \item There are three paramodular newforms of type $ \bf{(G)} $ attaining the largest index of the Hecke ring inside the maximal order of the Hecke field among all paramodular forms computed. These forms appear in levels $384$, $448$ and $483$, with the two first ones having Hecke field $\Q(\zeta_7)^+$, and the third $\Q(\sqrt{5})$. In all three cases, the index of the Hecke ring inside the maximal order is $32$. 
\end{itemize}



\subsection*{Future directions}

We close with some possible directions for future work.

\begin{itemize}
\item An immediate goal is to gather data for higher weights $(k,j)$ in a range where computation is feasible and find the Euler factors at the bad primes to have the complete eigensystem for every form.

\item To understand oldforms and newforms, another direction we may proceed is to implement the level-raising operators between the spaces of orthogonal modular forms and relating them to the corresponding level-raising map for paramodular forms given in \eqref{eqn:lvl1}--\eqref{eqn:lvl3}. This could be used to construct directly the spaces of newforms, without computing the exact multiplicity in which each oldform appears (or giving an independent check).

\item Finally, we also hope to complete the certification of the $L$-series for the newforms we have found, perhaps by further speeding up the computation, in order to allows us to compute a sufficient amount of Hecke eigenvalues. 
\end{itemize}

\bibliographystyle{amsalpha}


\end{document}